\documentclass[12pt,british,a4paper,reqno]{amsart}
\usepackage[T1]{fontenc}
	\usepackage[latin9]{inputenc}
	\usepackage[british]{babel}

\usepackage{paralist}

\newcommand{\Bigsum}[2]{\ensuremath{\mathop{\textstyle\sum}_{#1}^{#2}}}
\newcommand{\czero}{\ensuremath{\operatorname{c}_0}}
\newcommand{\columnvec}[2]{\genfrac{[}{]}{0pt}{}{\,#1\,}{#2}}
\DeclareMathSymbol{\shortminus}{\mathbin}{AMSa}{"39}
\usepackage{bbold}

\newcommand{\BAN}{\operatorname{\mathsf{BAN}}\nolimits}
\newcommand{\NOR}{\operatorname{\mathsf{NOR}}\nolimits}
\newcommand{\SNOR}{\operatorname{\mathsf{SNOR}}\nolimits}
\newcommand{\FRE}{\operatorname{\mathsf{FRE}}\nolimits}
\newcommand{\TVS}{\operatorname{\mathsf{TVS}}\nolimits}
\newcommand{\HDTVS}{\operatorname{\mathsf{HD-TVS}}\nolimits}
\newcommand{\LCS}{\operatorname{\mathsf{LCS}}\nolimits}
\newcommand{\HDLCS}{\operatorname{\mathsf{HD-LCS}}\nolimits}
\newcommand{\NUC}{\operatorname{\mathsf{NUC}}\nolimits}
\newcommand{\NUCFRE}{\operatorname{\mathsf{FN}}\nolimits}

\newcommand{\LB}{\operatorname{\mathsf{LB}}\nolimits}
\newcommand{\BOR}{\operatorname{\mathsf{BOR}}\nolimits}

\newcommand{\COM}{\operatorname{\mathsf{COM}}\nolimits}

\newcommand{\FS}{\operatorname{\mathsf{FS}}\nolimits}
\newcommand{\FH}{\operatorname{\mathsf{FH}}\nolimits}

	\usepackage{amsmath}
	\usepackage{amssymb}
	\usepackage{amsthm}
	\usepackage{bbm}
	\usepackage{braket}
	\usepackage{xcolor}
	\usepackage{enumitem}
	\usepackage{fancyhdr}
	\usepackage{geometry}
	\usepackage{graphicx}
	\usepackage[hyperfootnotes=false]{hyperref}
		\usepackage{bookmark}
	\usepackage{ifsym}
	\usepackage{indentfirst}
	\usepackage{mathrsfs}
	\usepackage{mathtools}
	\usepackage{parskip}
	\usepackage{proof}
	\usepackage{qtree}
	\usepackage{setspace}
	\usepackage{soul}
	\usepackage{tensor}
	\usepackage{tikz}
	\usepackage{tikz-cd}
	\usepackage{yhmath}

\usepackage{scalerel}

\makeatletter
	\@namedef{subjclassname@2020}{\textup{2020} Math.\ Subj.\ Class}

	\renewcommand{\andify}{%
		\nxandlist{\unskip, }{\unskip{} \@@and~}{\unskip \penalty-2 \space \@@and~}}
    
	\renewcommand\author@andify{%
  		\nxandlist {\unskip ,\penalty-1 \space\ignorespaces}%
		{\unskip {} \@@and~}%
		{\unskip \penalty-2 \space \@@and~}
	}
\makeatother

\DeclareRobustCommand{\myblacksquare}{%
  \makebox[.5em]{\vrule height 1.00ex depth -.15ex width 0.9ex}}

	\bookmarksetup{numbered,open}
	\renewcommand{\phi}{\varphi}
	\renewcommand{\Im}{\operatorname{Im}\nolimits}

	\providecommand{\corollaryname}{Corollary}
	\providecommand{\definitionname}{Definition}
	\providecommand{\examplename}{Example}
	\providecommand{\lemmaname}{Lemma}
	\providecommand{\propositionname}{Proposition}
	\providecommand{\remarkname}{Remark}
	\providecommand{\theoremname}{Theorem}
	\providecommand{\setupname}{Setup}
	\providecommand{\conjecturename}{Question}
	\providecommand{\scholiumname}{Scholium}
	\providecommand{\questionname}{Question}

	\theoremstyle{plain}
		\newtheorem{thm}{\protect\theoremname}[section] 
		\newtheorem{prop}[thm]{\protect\propositionname}
		
		\newtheorem{cor}[thm]{\protect\corollaryname}

	\theoremstyle{definition}
		\newtheorem{defn}[thm]{\protect\definitionname}

	\theoremstyle{definition}
		\newtheorem{rem}[thm]{\protect\remarkname}
		
	\numberwithin{figure}{section}
	\numberwithin{equation}{section}

	\newenvironment{acknowledgements}{
		\begin{abstract}} {\end{abstract}}

	\usetikzlibrary{matrix,arrows,decorations,decorations.pathmorphing,positioning,decorations.pathreplacing,shapes}
	\tikzset{commutative diagrams/.cd, 
		mysymbol/.style = {start anchor=center, end anchor = center, draw = none}}

	\newcommand{\commutes}[2][\circ]{\arrow[mysymbol]{#2}[description]{#1}}

	\newcommand{\BE}{\mathbb{E}}

	\newcommand{\CA}{\mathcal{A}}
	\newcommand{\CB}{\mathcal{B}}
	\newcommand{\CC}{\mathcal{C}}
	\newcommand{\CE}{\mathcal{E}}
	
	\newcommand{\CX}{\mathcal{X}}
	
	\newcommand{\Ker}{\operatorname{Ker}\nolimits}

	\newcommand{\Coim}{\operatorname{Coim}\nolimits}

	\newcommand{\obj}{\operatorname{obj}\nolimits}
	
	\newcommand{\Hom}{\operatorname{Hom}\nolimits}	
	\newcommand{\End}{\operatorname{End}\nolimits}

	\newcommand{\op}{\mathrm{op}}
	\newcommand{\rMod}[1]{\operatorname{\mathsf{Mod}}\nolimits{#1}}

\begin{document}
\title
[(Non-)examples of integral categories and the A.I.\ property]%
{Examples and non-examples of integral categories and the admissible intersection property}
\author[Hassoun]{Souheila Hassoun}
    \address{D\'{e}partement de math\'{e}matiques \\
        Universit\'{e} de Sherbrooke \\ 
        Sherbrooke, Qu\'{e}bec, J1K 2R1 \\ 
        Canada
        }
    \email{souheila.hassoun@usherbrooke.ca}
\author[Shah]{Amit Shah}
    \address{School of Mathematics, Statistics and Physics\\
        Newcastle University\\
        Newcastle upon Tyne, NE1 7RU\\
        United Kingdom
        }
    \email{amit.shah@newcastle.ac.uk}

\author[Wegner]{Sven-Ake Wegner}
    \address{Department of Mathematics\\
        University of Hamburg\\
        Bundesstra\ss{}e 55\\
        20146 Hamburg\\
        Germany
        }
    \email{sven.wegner@uni-hamburg.de}

\date{\today\vspace{2pt}}
\subjclass[2020]{%
Primary 18E05; 
Secondary 18E10, 46A08, 46A13, 46A17, 46M10, 46M15%
\vspace{1.5pt}}
\dedicatory{Dedicated to Prof.\ Wolfgang Rump on the occasion of his $70$th birthday}

\begin{abstract}
%
%
%
Integral categories form a sub-class of pre-abelian categories whose systematic study was initiated by Rump in 2001. 
In the first part of this article we determine whether several categories of topological and bornological vector spaces are integral.
Moreover, we establish that the class of integral categories is not contained in the class of quasi-abelian categories, and that there exist semi-abelian categories that are neither integral nor quasi-abelian. 
In the last part of the article we show that a category is quasi-abelian if and only if it has admissible intersections, in the sense considered recently by Br{\"u}stle, Hassoun and Tattar. 
This exhibits that a rich class of non-abelian categories having this property arises naturally in functional analysis.
\end{abstract}

\maketitle


\section{Introduction}
\label{SEC:INTRO}

Since the 1960s there has been much research on additive, non-abelian categories. This has led to the development of a spectrum of classes of categories ranging from pre-abelian to abelian. Our goal in this note is to explain the following diagram:\label{PIC}\vspace{0.5cm}

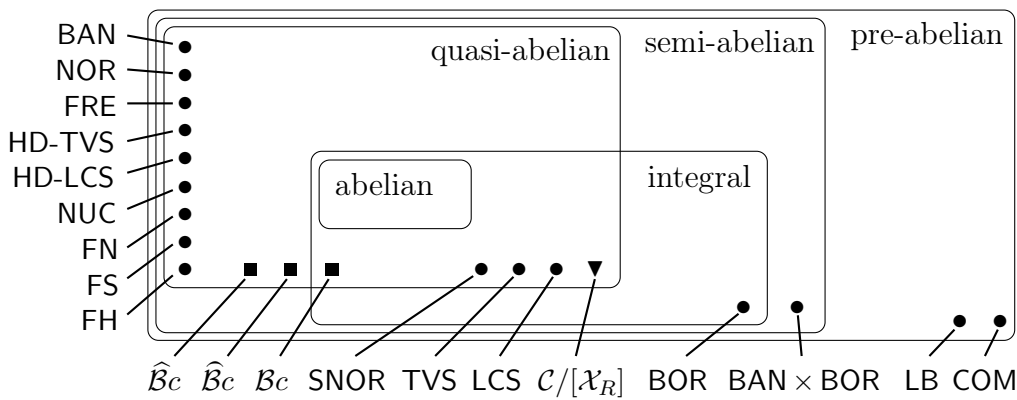
\[
\makebox[\textwidth][c]{
\hspace{-20pt}
\begin{picture}(300,100)(0,0)

\put(0,0){
\begin{tikzpicture}
\draw [fill=white!20, rounded corners](0,0) rectangle (11.4,4.38);
\end{tikzpicture}
}
\put(267,112){pre-abelian}

\put(3,3){
\begin{tikzpicture}
\draw [fill=white!20, rounded corners](0,0) rectangle (8.8,4.17);
\end{tikzpicture}
}
\put(190,109){semi-abelian}

\put(6,20){
\begin{tikzpicture}
\draw [fill=white!20, rounded corners](0,0) rectangle (6,3.46);
\end{tikzpicture}
}
\put(110,106){quasi-abelian}

\put(61,6){
\begin{tikzpicture}
\draw [rounded corners](0,0) rectangle (6,2.3);
\end{tikzpicture}
}
\put(191,58){integral}

\put(64,42.25){
\begin{tikzpicture}
\draw [fill=white!20, rounded corners](0,0) rectangle (2,0.91);
\end{tikzpicture}
}
\put(74,55){abelian}

\put(15,108){$\bullet$}\put(-30,112){\small$\BAN$}
\put(15,97.5){$\bullet$}\put(-30.6,98.5){\small$\NOR$}
\put(15,87){$\bullet$}\put(-27.5,85){\small$\FRE$}
\put(15,76.5){$\bullet$}\put(-48.5,71.5){\small$\HDTVS$}
\put(15,66){$\bullet$}\put(-46.75,58){\small$\HDLCS$}
\put(15,55){$\bullet$}\put(-30,44.5){\small$\NUC$}
\put(15,45){$\bullet$}\put(-20.9,31){\small$\NUCFRE$}
\put(15,34.5){$\bullet$}\put(-19,17.5){\small$\FS$}
\put(15,24){$\bullet$}\put(-21,4){\small$\FH$}

\put(39.5,24){$\myblacksquare$}
    \put(4,-18.5){\small$\widehat{\mathcal{B}}c$}
\put(54.5,24){$\myblacksquare$}
    \put(24,-18.5){\small$\wideparen{\mathcal{B}}c$}
\put(70,24){$\myblacksquare$}
    \put(44,-18.5){\small$\mathcal{B}c$}

\put(126,24){$\bullet$}
    \put(64,-18.5){\small$\SNOR$}
\put(140,24){$\bullet$}
    \put(99,-18.5){\small$\TVS$}
\put(154,24){$\bullet$}
    \put(125,-18.5){\small$\LCS$}
\put(168,24.5){$\scalerel*{\blacktriangledown}{\bullet}$}
    \put(150,-18.5){{\small$\CC/[\CX_{R}]$}}

\put(224,10){$\bullet$}
    \put(191,-18.5){\small$\BOR$}

\put(244,10){$\bullet$}
    \put(221,-18.5){{\small$\BAN\times\BOR$}}


\put(305,4.5){$\bullet$}
    \put(287,-18.5){\small$\LB$}
\put(320,4.5){$\bullet$}
    \put(305,-18.5){\small$\COM$}

\put(-8,-6){
\begin{tikzpicture}
  \phantom{\draw[style=help lines] (-1,0) grid (11,5);}
    
    \draw[thick] (-1,4.25) -- (-0.4,4.14); 
    \draw[thick] (-1,3.8) -- (-0.4,3.75); 
    \draw[thick] (-1,3.35) -- (-0.4,3.35); 
    \draw[thick] (-1,2.85) -- (-0.4,2.98); 
    \draw[thick] (-1,2.43) -- (-0.4,2.6); 
    \draw[thick] (-1,1.95) -- (-0.4,2.2); 
    \draw[thick] (-1,1.44) -- (-0.4,1.85); 
    \draw[thick] (-1,1.03) -- (-0.4,1.48); 
    \draw[thick] (-1,0.55) -- (-0.4,1.07); 

    \draw[thick] (-0.38,0) -- (0.55,1.02); 
  \draw[thick] (0.32,0) -- (1.1,1.02); 
  \draw[thick] (1.01,0) -- (1.65,1.02); 

  \draw[thick] (2.0,0) -- (3.55,1.06); 
  \draw[thick] (3.05,0) -- (4.08,1.05); 
  \draw[thick] (3.9,0) -- (4.6,1.05); 
  \draw[thick] (4.9,0) -- (5.15,1.02); 
  
  \draw[thick] (6.35,0) -- (7.04,0.56); 
  \draw[thick] (7.87,0) -- (7.82,0.54); 
    
  \draw[thick] (9.55,0) -- (9.88,0.35); 
  \draw[thick] (10.28,0) -- (10.45,0.35); 
\end{tikzpicture}
}

\end{picture}
}\vspace{0.8cm}
\]
We focus on the concrete examples from functional analysis and representation theory that appear therein. 
We refer the reader to \S\S\ref{SEC:PRE-ABELIAN-CATEGORIES}--\ref{SEC:BVS} for precise definitions.

A \emph{pre-abelian} category is an additive category in which every morphism has a kernel and a cokernel. 
Within the class of pre-abelian categories, one defines the following notions. 
\emph{Semi-abelian} categories are the ones in which the canonical morphism between the coimage and image is always both monic and epic, but not necessarily an isomorphism as one would expect in an abelian category. 
\emph{Quasi-abelian} categories are the ones in which kernels are stable under pushout and cokernels are stable under pullback. 
On the other hand, \emph{integral} categories are the ones in which monomorphisms are stable under pushout and epimorphisms are stable under pullback. 
One can check that the implications
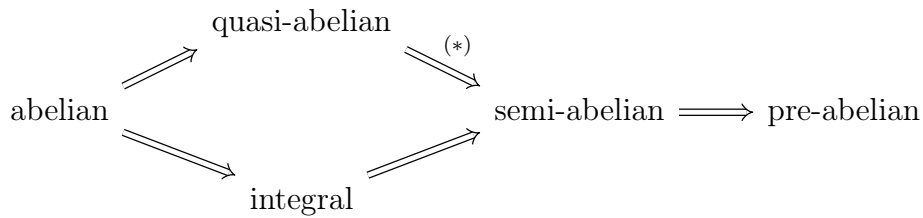
\[
\begin{tikzcd}[row sep=0.5cm]
&\text{quasi-abelian}\arrow[Rightarrow, start anchor=south east, end anchor=north west,pos=0.4]{dr}{(*)}&&\\
\text{abelian} \arrow[Rightarrow, start anchor=north east, end anchor=south west]{ur}\arrow[Rightarrow, start anchor=south east, end anchor=north west]{dr}&& \text{semi-abelian} \arrow[Rightarrow]{r}&\text{pre-abelian}\\
&\text{integral}\arrow[Rightarrow, start anchor=north east, end anchor=south west]{ur}&&
\end{tikzcd}
\]
are relatively straightforward; see e.g.\ Rump \cite{Rump01-almost-abelian-categories}. 
A conjecture of Ra\u{\i}kov, which has been solved in the negative, was that the converse of $(*)$ holds; see Remark \ref{REM-RAIKOV-BOR}.

Despite recent progress on integral categories, which appears to be predominantly in algebra (see Remark \ref{rem:on-integral-categories}), the question if there exist integral categories that are not quasi-abelian seems to date to be open. 
We answer this positively; see Corollary \ref{cor:integral-not-contained-in-quasi-abelian}. 
With an idea communicated to the authors by J.\ Wengenroth, we also prove that the class of semi-abelian categories is not merely the union of the classes of integral and quasi-abelian categories; see Theorem \ref{thm:semi-category-that-is-not-integral-or-quasi-abelian}. 
Furthermore, we systematically investigate integrality for many examples found in the functional analyst's category theory toolbox; see Theorems \ref{THM-0}--\ref{THM-BOR}, \ref{THM-2} and \ref{THM-3}. 
As a consequence of these results, we derive that most of the categories in the diagram on p.\ \pageref{PIC} have neither enough projectives nor enough injectives; see Theorem \ref{thm:categories-dont-have-enough-quasi-projs-injs}.

Non-abelian categories appear in abundance in functional analysis and have applications for instance in the theory of partial differential equations; see Wengenroth \cite{Wengenroth03}, and Frerick and Sieg \cite{FrerickSieg}, and the references therein. 
Indeed, as can be seen from the diagram on p.\ \pageref{PIC}, most of the categories we study here are quasi-abelian but not abelian. 
However, this is still enough intrinsic structure to conduct homological algebra as Schneiders \cite{Schneiders} did. 
He also observed that on each quasi-abelian category the class of all kernel-cokernel pairs forms an \emph{exact structure} in the sense of Quillen \cite{Quillen} (see also Yoneda's `quasi-abelian $\mathscr{S}$-categories' \cite{Yoneda61}). 
In contrast to the internal structure of a category, like pre-, semi- and quasi-abelian, an exact structure is extrinsic. 

In studying lengths of objects in exact categories, Br{\"u}stle, Hassoun, Langford and Roy \cite[Exam.\ 6.9]{BHLR-reduction} showed that an analogue of the classic Jordan-H{\"o}lder property can fail for an arbitrary exact category; see also Enomoto \cite{Enomoto}. 
Motivated partly by this, Br{\"u}stle, Hassoun and Tattar \cite{BHT} have recently considered additive categories with a mix of intrinsic and extrinsic structures. 
More specifically, they consider pre-abelian categories equipped with an exact structure that has `admissible intersections'; see \S\ref{SEC:AIP}. 
Building on their groundwork, we show in Theorem \ref{thm:AI-iff-quasi-abelian} that this property is satisfied if and only if the category is quasi-abelian, thereby giving a new characterisation for quasi-abelian categories.

\medskip

\section{A reminder on pre-abelian categories}
\label{SEC:PRE-ABELIAN-CATEGORIES}

We recall some definitions of additive categories more general than abelian ones. 
For more details we refer the reader to \cite{Rump01-almost-abelian-categories}. 
Recall that an additive category is called \emph{pre-abelian} if every morphism has a kernel and a cokernel. 

For the remainder of this section, let $\CA$ be a pre-abelian category.
\begin{defn}
\label{def:semi-abelian-category}\cite[p.\ 167]{Rump01-almost-abelian-categories}
If each morphism $f\colon A\to B$ in $\CA$ can be expressed as $f=i\circ{}p$ for some monomorphism $i$ and some cokernel $p$, then $\CA$ is said to be \emph{left semi-abelian}.
Dually, if each morphism $f$ can be written as $f=i\circ{}p$ for some kernel $i$ and some epimorphism $p$, then $\CA$ is said to be \emph{right semi-abelian}.
If $\CA$ is both left and right semi-abelian, then it is called \emph{semi-abelian}.
\end{defn}
\begin{defn}
\label{def:stable-under-PB-or-PO}
Let $\CX$ be a class of morphisms in $\CA$. 
We say that $\CX$ is \emph{stable under pullback} if, in any pullback square 
\begin{equation*}
\label{eqn:stable-under-PO-PB}
\begin{tikzcd}
A\arrow{r}{a}\arrow{d}[swap]{b} \commutes[\text{PB}]{dr}& B \arrow{d}{c}\\ C \arrow{r}[swap]{d} & D
\end{tikzcd}
\end{equation*}
$a$ is in $\CX$ whenever $d$ is in $\CX$. 
Being \emph{stable under pushout} is defined dually.
\end{defn}
\begin{defn}
\label{def:quasi-abelian-category}\cite[p.\ 168]{Rump01-almost-abelian-categories}
If cokernels in $\CA$ are stable under pullback, then $\CA$ is called \emph{left quasi-abelian}.
Dually, if kernels in $\CA$ are stable under pushout, then $\CA$ is called \emph{right quasi-abelian}. 
If $\CA$ is both left and right quasi-abelian, then it is called \emph{quasi-abelian}.
\end{defn}
We note here that different authors, in particular Palamodov \cite{Palamodov68, Palamodov71} and Ra\u{\i}kov \cite{Raikov}, have used the above notions in different senses. 
For example, we follow Palamodov in the use of `semi-abelian'; see also Kopylov and Wegner \cite{KopylovWegner12} for different characterisations. 
In the non-additive setting, this same name is used to describe a category that is pointed Barr-exact protomodular, admitting binary coproducts; see Janelidze, M{\'a}rki and Tholen \cite{JanelidzeMarkiTholen}.
We refer to the introductions of \cite{Rump11}, \cite{KopylovWegner12}, and \cite{Wengenroth12} for historic references. 
\begin{defn}
\label{def:integral-category}\cite[p.\ 168]{Rump01-almost-abelian-categories}
If epimorphisms in $\CA$ are stable under pullback, then $\CA$ is called  \emph{left integral}. 
Dually, if monomorphisms in $\CA$ are stable under pushout, then $\CA$ is called \emph{right integral}. 
If $\CA$ is both left and right integral, then it is called \emph{integral}.
\end{defn}
Certain relationships between the categories defined above can then be established. 
\begin{prop}\label{prop:quasi-abelian-or-integral-implies-semi-abelian}
\cite[p.\ 169, Cor.\ 1]{Rump01-almost-abelian-categories} \vspace{-4pt}
\begin{compactitem}
    \item[\emph{(i)}] If $\CA$ is a left (respectively, right) quasi-abelian category, then it is left (respectively, right) semi-abelian.\vspace{3pt} 
    \item[\emph{(ii)}] If $\CA$ is a left (respectively, right) integral category, then it is left (respectively, right) semi-abelian.
\end{compactitem}
\end{prop}
It follows then that the classes of quasi-abelian and integral categories are both contained in the class of semi-abelian categories.
\begin{prop}
\label{prop:semi-abelian-implies-left-right-equivalence}
\emph{\cite[Prop.\ 3 and p.\ 173, Cor.]{Rump01-almost-abelian-categories}}
Suppose $\CA$ is semi-abelian. \vspace{-4pt}
\begin{compactitem}
    \item[\emph{(i)}] 
        The category $\CA$ is left quasi-abelian if and only if it is right quasi-abelian.\vspace{3pt}
    \item[\emph{(ii)}]  
        The category $\CA$ is left integral if and only if it is right integral.
\end{compactitem}
\end{prop}

We conclude this section with the following remark on integral categories. 
\begin{rem}\label{rem:on-integral-categories}
Although Rump introduced the name `integral' for a category, such categories were known to B{\u{a}}nic{\u{a}} and Popescu \cite{BanicaPopescu}. 
By extending results of \cite{BanicaPopescu}, Rump proved that a pre-abelian category is integral if and only if it admits a faithful embedding into an abelian category which preserves kernel-cokernel pairs; see \cite[Prop.\ 7]{Rump01-almost-abelian-categories}.
In particular, he observed that the class of simultaneously monic and epic morphisms in an integral category admits a \emph{calculus of fractions} in the sense of Gabriel and Zisman \cite{GabrielZisman}, and hence the (canonical) localisation of the category at this class is an abelian category. 
This is one reason why integral categories have gained popularity among representation theorists:\vspace{-4pt}
\begin{compactitem}
\item[(i)] Rump \cite{Rump01-star-modules-tilting, Rump04-categories-of-lattices, Rump04-differentiation, Rump06} himself showed, among other things, that the torsion-free class of a hereditary torsion theory in an abelian category is integral. \vspace{3pt} 

\item[(ii)] Buan and Marsh \cite{BuanMarsh} showed that, for a certain triangulated category $\CC$ and a rigid object $R\in\CC$, the quotient category $\CC/[\CX_{R}]$, where $\CX_R = \Ker\Hom_{\CC}(R,-)$, is integral. 
From this they proved that the canonical localisation of $\CC/[\CX_{R}]$ is a module category over $(\End_{\CC}R)^{\op}$. \vspace{3pt} 

\item[(iii)] By introducing \emph{hearts} of twin cotorsion pairs on triangulated categories, Nakaoka \cite{Nakaoka} generalised this construction of $\CC/[\CX_{R}]$. 
He showed that the heart is always semi-abelian and gave a sufficient condition for it to be integral. 
A condition for the heart to be quasi-abelian was given by Shah \cite{Shah}. 
Furthermore, analogous concepts have been studied by Liu \cite{LiuY} for exact categories, and by Liu and Nakaoka \cite{LiuYNakaoka}, and Hassoun and Shah \cite{HassounShah} for \emph{extriangulated} categories (in the sense of Nakaoka and Palu \cite{NakaokaPalu}). 
\end{compactitem}
\end{rem}
%

\medskip


\section{Categories of topological vector spaces}
\label{SEC:TVS}

In this section we look at categories of \emph{topological vector spaces}. The objects of such a category are pairs $(X,\tau)$, where $X$ is a vector space and $\tau$ is a topology on $X$ that makes the vector space operations continuous. The morphisms are continuous linear maps. 
For unexplained notation from functional analysis we refer the reader to Meise and Vogt \cite{MeiseVogt}.

Our first result extends Rump's observation \cite[\S 2.2]{Rump01-almost-abelian-categories} that the topological abelian groups form an integral category.

\begin{thm}\label{THM-0} Let $k\in\{\mathbb{R},\mathbb{C}\}$ be fixed. The categories \vspace{-4pt}
\begin{compactitem}
\item[\emph{(i)}]  $\SNOR$ of semi-normed spaces;\vspace{3pt}
\item[\emph{(ii)}] $\LCS$ of (Hausdorff and non-Hausdorff) locally convex spaces; and\vspace{3pt}
\item[\emph{(iii)}]$\TVS$ of (Hausdorff and non-Hausdorff) topological vector spaces\vspace{-4pt}
\end{compactitem}
over $k$, each furnished with linear and continuous maps as morphisms, are quasi-abelian and integral.
\end{thm}
\begin{proof}
It is well-known that all three categories are quasi-abelian; see e.g.\ \cite[Prop.\ 3.2.4]{Schneiders}, Prosmans \cite[Prop.\ 2.1.11]{Prosmans00} and \cite[Exam.\ 4.14]{FrerickSieg}. 
In $\SNOR$, $\LCS$ and $\TVS$ the kernel of a morphism $f\colon X\rightarrow Y$ is the inclusion $f^{-1}(0)\rightarrow X$, where $f^{-1}(0)$ is furnished with the induced topology. 
Denote by $\operatorname{ran}f$ the range of $f$. 
Then the cokernel of $f$ is the quotient $Y\rightarrow Y/\operatorname{ran}f$ with the quotient topology; see e.g.\ \cite[Lem.\ 3.2.3]{Schneiders}, \cite[Prop.\ 2.1.8]{Prosmans00} and \cite[Exam.\ 2.14]{FrerickSieg}. 
Thus, $f$ is monic if and only if $f$ is injective, and $f$ is epic if and only if $f$ is surjective. Since pushouts and pullbacks compute algebraically precisely as in $\rMod{k}$, the two conditions in Definition \ref{def:integral-category} hold.
\end{proof}

Our second result exhibits a collection of quasi-abelian categories that are not integral. This is due to the Hausdorff property that we require below. Although all categories in Theorem \ref{THM-1} are full subcategories of $\TVS$, their cokernels and thus pushouts compute algebraically differently than in $\rMod{k}$. Theorem \ref{THM-1} extends Rump's results \cite[\S 2.2]{Rump01-almost-abelian-categories} on Hausdorff topological abelian groups.

\begin{thm}\label{THM-1} Let $k\in\{\mathbb{R},\mathbb{C}\}$ be fixed. 
The categories\vspace{-4pt}
\begin{compactitem}
\item[\emph{(i)}] $\BAN$ of Banach spaces;\vspace{3pt}
\item[\emph{(ii)}] $\NOR$ of normed spaces;\vspace{3pt}
\item[\emph{(iii)}] $\FRE$ of Fr{\'e}chet spaces;\vspace{3pt}
\item[\emph{(iv)}] $\HDLCS$ of Hausdorff locally convex spaces;\vspace{3pt}
\item[\emph{(v)}] $\HDTVS$ of Hausdorff topological vector spaces;\vspace{3pt}
\item[\emph{(vi)}] $\NUC$ nuclear spaces;\vspace{3pt}
\item[\emph{(vii)}] $\NUCFRE$ of nuclear Fr{\'e}chet spaces;\vspace{3pt}
\item[\emph{(viii)}] $\FS$ of Fr{\'e}chet-Schwartz spaces; and\vspace{3pt}
\item[\emph{(ix)}] $\FH$ of Fr{\'e}chet-Hilbert spaces\vspace{-4pt}
\end{compactitem}
over $k$, each furnished with linear and continuous maps as morphisms, are quasi-abelian but not (left or right) integral.
\end{thm}
\begin{proof} Again, it is well-known that these categories are quasi-abelian; see e.g.\ Prosmans \cite[Prop.\ 3.1.7]{Prosmans95}, \cite[Prop.\ 3.2.17]{Schneiders}, \cite[Prop.\ 4.4.5]{Prosmans00} and \cite[Prop.\ 3.1.8]{Prosmans00} for direct proofs for the first four. 
The most efficient approach, however, is to establish explicitly that $\HDTVS$ is quasi-abelian, which can be achieved with a slight modification of the proofs just cited. 
In doing so, one observes that given a morphism $f\colon X\rightarrow Y$ in $\HDTVS$, its kernel is the inclusion $f^{-1}(0)\rightarrow X$, and its cokernel is the quotient $Y\rightarrow Y/\overline{\operatorname{ran}f}$. 
These spaces are endowed with the subspace and the quotient topology, respectively. 
Since the defining properties for the other categories, like Banach, normed, Fr{\'e}chet, etc., are inherited by closed subspaces and quotients by closed subspaces\footnote{For the not-so-explicitly-studied categories in (vi)--(ix), this can be found in \cite[Prop.\ 28.6, Prop.\ 24.18 and Rmk.\ 29.15]{MeiseVogt}.}, these categories reflect the kernels and cokernels of $\HDTVS$. From this it follows that all these categories are also quasi-abelian by, for example, \cite[Prop.\ 4.20]{FrerickSieg}.

By Propositions \ref{prop:quasi-abelian-or-integral-implies-semi-abelian} and \ref{prop:semi-abelian-implies-left-right-equivalence}, left integrality is equivalent to right integrality for all categories in our list; thus, below we show that they all are not right integral. 
For this we bear in mind that in all nine categories, a morphism is monic if and only if it is injective. This follows from our observations above about kernels in these categories.

(i)--(v):\; Consider the Banach spaces 
\[
\czero{}=\displaystyle\bigl\{x=(x_{j})_{j\in\mathbb{N}}\in k^{\mathbb{N}}\:\big|\,\lim_{j\rightarrow\infty}x_j=0\bigr\} \;\;\text{ and } \;\; \ell^1=\displaystyle\bigl\{x\in k^{\mathbb{N}}\:\big|\:\|x\|_1=\Bigsum{j=1}{\infty}|x_j|<\infty\bigr\}
\]
of null sequences and of absolutely summable sequences, respectively. Here, $\czero$ is endowed with the supremum norm given by $\|x\|_{\infty}=\sup_{j\in\mathbb{N}}|x_j|$ and $\ell^1$ is endowed with the 1-norm $\|\cdot\|_1$ indicated above. The field $k$ is endowed with the absolute value as a norm. We denote by $i\colon\ell^1\rightarrow\czero$ the inclusion and by $\Sigma\colon\ell^1\rightarrow k$ the map that sends a sequence to its sum. Now we put
$
P=(\czero\oplus\,k)/\,\textstyle\overline{\operatorname{ran}\scriptstyle\columnvec{\phantom{\shortminus}i}{\shortminus\Sigma}}
$, 
where $\czero\oplus\,k$ carries the product topology, the closure is taken in $\czero\oplus\,k$, and $P$ is furnished with the quotient topology. We denote by $p\colon\czero\oplus\,k\rightarrow P$ the quotient map, and by $i_1\colon \czero\rightarrow \czero\oplus\,k$ and $i_2\colon k\rightarrow \czero\oplus\,k$ the inclusion maps. We claim that in all five categories the diagram
\begin{equation*}
\begin{tikzcd}
\ell^1\arrow[d, swap, "\Sigma"]\arrow[r, "i"]& \czero \arrow[d, "p\hspace{0.75pt}\circ\hspace{0.75pt}i_1"]\\
k\arrow[r, swap, "p\hspace{0.75pt}\circ\hspace{0.75pt}i_2"]& P
\end{tikzcd}
\end{equation*}
is a pushout square, and that $i$ is a monomorphism but $p\circ i_2$ is not.

Since the pushout of $i$ along $\Sigma$ is the cokernel of {\small$\columnvec{\phantom{\shortminus}i}{\shortminus\Sigma}$}$\colon\ell^1\rightarrow\czero\oplus\,k$, our initial remarks establish the first claim and imply that for the second claim it is enough to show that $p\circ i_2$ is not injective. In order to achieve this we will establish that $(p\circ i_2)(1) = 0$ in $P$. 
Applying the definition of $p\circ i_2$, we see that we need to show that
\[
{\textstyle\columnvec{0}{1}}\in\overline{\big\{{\textstyle\columnvec{\phantom{\shortminus}i}{\shortminus\Sigma}}(x)\:\big|\:x\in\ell^1\bigr\}}
\] 
holds. For this we define a sequence $(x^n)_{n\in\mathbb{N}}$ in $\ell^1$ as follows. For positive integers $n$ and $j$ we put $x^n_j=-1/n$ whenever $1\leqslant j\leqslant n$, and $x^n_j=0$ otherwise. Since for each $n$ only finitely many entries of
\[
x^n=(-1/n,-1/n,\dots,-1/n,0\dots)
\]
are non-zero, we get $(x^n)_{n\in\mathbb{N}}\subseteq\ell^1$. In view of $\|x^n\|_{\infty}=1/n$ and $i(x^n)=x^n$ we see that $(i(x^n))_{n\in\mathbb{N}}$ converges to $0$ in $\czero$. 
On the other hand, we have
\[
\bigl|1-(-\Sigma(x^n))\bigr|=\bigl|1+\Bigsum{j=1}{n}{-1/n}\bigr| = 0
\]
for every $n$. Whence, $(-\Sigma(x^n))_{n\in\mathbb{N}}$ converges to $1$ in $k$
and {\small$\columnvec{0}{1}$}$\in\overline{\operatorname{ran}\scriptstyle\columnvec{\phantom{\shortminus}i}{\shortminus\Sigma}}$, as desired.

(vi)--(ix):\; 
Since nuclear Fr{\'e}chet spaces are Fr{\'e}chet-Hilbert and Fr{\'e}chet-Schwartz by \cite[Lem.\ 28.1 and Cor.\ 28.5]{MeiseVogt}, we construct a pushout diagram like in the first part of the proof but with all spaces being nuclear Fr{\'e}chet. 
As a locally convex space is simultaneously Banach and nuclear if and only if it is of finite dimension, we need nuclear replacements for $\czero$ and $\ell^1$. 
First, consider the space $k^{\mathbb{N}}$ of all sequences, which carries the topology of point-wise convergence given by $|x|_m=\sup_{1\leqslant j\leqslant m}|x_j|$. Secondly, let 
\[
s=\bigl\{x\in k^{\mathbb{N}}\:\big|\:\forall\:m\in\mathbb{N}\colon \|x\|_m=\Bigsum{j=1}{\infty}j^{m}|x_j|<\infty\bigr\}
\]
be the space of rapidly decreasing sequences, which is endowed with the topology generated by the semi-norms $(\|\cdot\|_m)_{m\in\mathbb{N}}$. 
Both spaces are nuclear\footnote{This follows from \cite[Prop.\ 28.16]{MeiseVogt} since $s=\lambda^1((j^{m})_{j,m})$ and $k^{\mathbb{N}}=\lambda^{\infty}((\mathbb{1}_{\{1,\dots,m\}}(j))_{j,m})$ are both from the class of K{\"o}the echelon spaces.}. 
Since $|x|_m\leqslant\|x\|_m$ and $|\Sigma(x)|\leqslant\|x\|_m$ hold for every $m\in\mathbb{N}$ and every $x\in s$, the inclusion $i'\colon s\rightarrow k^{\mathbb{N}}$ and the summation $\Sigma'\colon s\rightarrow k$ are both well-defined and continuous. 
The space 
$P' = ( k^{\mathbb{N}} \oplus k )/
\,\overline{\operatorname{ran}\scriptstyle\columnvec{\phantom{\shortminus\shortminus}i'}{\shortminus\Sigma'}}$ and the quotient map  $p'\colon k^{\mathbb{N}} \oplus k\rightarrow P'$ are defined analogously to the first part.

Our observations at the beginning of this proof imply that the following holds in all four categories. 
Firstly, the pushout of $i'$ along $\Sigma'$ is given by the cokernel of {\small$\columnvec{\phantom{\shortminus\shortminus}i'}{\shortminus\Sigma'}$}, and thus precisely by $P'$. Secondly, $i'$ is monic.
Therefore the diagram 
\[
\begin{tikzcd}
s\arrow[d, swap, "\Sigma'"]\arrow[r, "i'"]& k^{\mathbb{N}} \arrow[d, "p'\hspace{0.75pt}\circ\hspace{0.75pt}i_{1}'"]\\
k\arrow[r, swap, "p'\hspace{0.75pt}\circ\hspace{0.75pt}i_{2}'"]& P'
\end{tikzcd}
\]
is a pushout. By employing the same sequence $(x^n)_{n\in\mathbb{N}}$ as in the first part, we can see that $p'\circ i_{2}'$ is not monic.
\end{proof}

We now consider two examples of categories that are neither semi-abelian nor integral. 
Both are full subcategories of $\HDLCS$; the first reflects the kernels and the second the cokernels of $\HDLCS$. 
However, in the first one cokernels compute differently than in $\HDLCS$, and in the second the kernels do.

\begin{thm}
\label{THM-COM}
Let $k\in\{\mathbb{R},\mathbb{C}\}$ be fixed. The category $\COM$ of complete Hausdorff locally convex spaces over $k$, furnished with linear and continuous maps as morphisms, is right quasi-abelian but neither left semi-abelian nor right integral.
\end{thm}
\begin{proof}
The category $\COM$ reflects kernels of $\HDLCS$, and thus a morphism $f\colon X\rightarrow Y$ in $\COM$ is monic if and only if it is injective. To compute the cokernel
\[
Y\rightarrow \widehat{Y/\overline{\operatorname{ran}f}}
\]
of $f$ in $\COM$, one must take a completion. 
From this it was derived in \cite[Exam.\ 4.2]{KopylovWegner12} that $\COM$ is not left semi-abelian. 
If, on the other hand, we go through the first example in the proof of Theorem \ref{THM-1}, but in the category $\COM$, we see that $P$ is already complete since we are dealing with Banach spaces. The diagram constructed in the proof of Theorem \ref{THM-1} is thus also a pushout in $\COM$, and hence $\COM$ is not right integral.

It remains to see that $\COM$ is right quasi-abelian. We remark that by \cite[Prop.\ 4.1.10 and Cor.\ 2.1.9]{Prosmans00}, a morphism in $\COM$ is a kernel if and only if it is injective and open onto its range. Notice that this implies automatically that $\operatorname{ran}f$ is closed; see \cite[Rmk.\ 4.1.11(i)]{Prosmans00}. Let $f\colon X\rightarrow Y$ be a kernel and $g\colon X\rightarrow Z$ be an arbitrary morphism. We put now 
$
Q= (Y\hspace{-1.5pt}\oplus\hspace{-1pt} Z) / \textstyle\overline{\operatorname{ran}\scriptstyle\columnvec{\phantom{\shortminus}g}{\shortminus{}f}}
$.
Notice that the pushout of $f$ along $g$ taken in $\COM$ factors through the pushout taken in $\HDLCS$. 
Thus, there is a diagram 
\[
\begin{tikzcd}
X\arrow[d, swap, "g"]\arrow[r, "f"]& Y \arrow[d]\\
Z\arrow[r, swap, "q\hspace{0.75pt}\circ\hspace{0.75pt}i_2"]\ar[equal]{d}& Q\arrow[d, "i"]\\
Z\arrow[r, swap, "i\hspace{0.75pt}\circ\hspace{0.75pt}q\circ\hspace{0.75pt}i_2"]& \widehat{Q}
\end{tikzcd}
\]
in which the outer rectangle is the pushout in $\COM$ and the upper square is the pushout in $\HDLCS$.
Here, $i_2\colon Z\rightarrow Y\oplus Z$ is the inclusion, $q\colon Y\oplus Z\rightarrow Q$ is the quotient map and $i\colon Q\rightarrow\widehat{Q}$ is the inclusion of $Q$ into its completion. 
Since $\HDLCS$ is quasi-abelian, $q\circ i_2$ is a kernel in $\HDLCS$, and thus injective and open onto its range; see \cite[Cor.\ 3.1.5]{Prosmans00}. 
Since $i$ is an isomorphism onto its range, we see that $i\circ q\circ i_2$ is injective and open onto its range, too. 
Thus, it is a kernel in $\COM$, and we are done.
\end{proof}

\begin{thm}\label{THM-LB} Let $k\in\{\mathbb{R},\mathbb{C}\}$ be fixed. The category $\LB$ of countable Hausdorff locally convex inductive limits of Banach spaces over $k$, furnished with linear and continuous maps as morphisms, is left quasi-abelian but neither right semi-abelian nor left integral.
\end{thm}
\begin{proof}
The category $\LB$ reflects cokernels of $\HDLCS$, and thus a morphism $f\colon X\rightarrow Y$ is epic in $\LB$ if and only if it has dense range. 
Forming the kernel
\[
f^{-1}(0)^{\flat}\rightarrow X
\]
of $f$ requires endowing $f^{-1}(0)$ with a possibly strictly finer topology; see Wegner \cite[Proof of Prop.\ 14]{Wegner}). 
The proof in \cite{Wegner} shows that $\LB$ is left semi-abelian but not semi-abelian---and therefore necessarily not right semi-abelian. 

Furthermore, $\LB$ is left quasi-abelian, as noted without proof already in \cite[p.\ 540]{KopylovWegner12}. Indeed, first observe that since $\LB$ reflects cokernels of $\HDLCS$, and since every cokernel is the cokernel of its own kernel, all cokernels of $\LB$ are surjective. 
Conversely, if $f\colon X\rightarrow Y$ is a surjective morphism in $\LB$, then it satisfies the universal property of a cokernel. 
Assume now that $f$ is a cokernel and let $g\colon Z\rightarrow Y$ be an arbitrary morphism in $\LB$. Then the pullback of $f$ along $g$ is 
\begin{equation*}
\begin{tikzcd}[column sep=21pt,row sep=21pt]
[g\;\shortminus\hspace{-4pt}f]^{-1}(0)^{\flat}\arrow[d, swap, "i_2"]\arrow[r, "i_1"]\commutes[\mathrm{PB}]{dr}& Z \arrow[d, "g"]\\
X\arrow[r, swap, "f"]& Y
\end{tikzcd}
\end{equation*}
which is algebraically the pullback taken in $\rMod{k}$. Thus, we see that $i_1$ is surjective and hence a cokernel in $\LB$ by the argument just above.

Finally, we use that $\BAN$ is a subcategory of $\LB$, in order to show that the latter is not left integral. Since $\BAN$ is not left integral by Theorem \ref{THM-1}, we can find a pullback diagram in $\BAN$ such that the bottom morphism is epic but the top one is not. Since for a Banach space $X$ and a closed subspace $U\subseteq X$, the topology of $U^{\flat}$ coincides with the topology induced by $X$, this diagram is also a pullback in $\LB$. 
From this we see that $\LB$ is not left integral.
\end{proof}

In view of Proposition \ref{prop:quasi-abelian-or-integral-implies-semi-abelian}, we note that it follows from Theorem \ref{THM-COM} that $\COM$ cannot be left quasi-abelian or left integral. Similarly, using Theorem \ref{THM-LB}, $\LB$ cannot be right quasi-abelian or right integral.

So far we have witnessed that there exist examples of quasi-abelian categories that are not integral. 
Next we give an example of an integral category that is not quasi-abelian. 
This establishes that the class of integral categories is not contained in the class of quasi-abelian ones. 
To the knowledge of the authors this seemed to be previously unknown. 
Notice that the cokernels appearing below have no closure in the denominator, since we deal here again with a category whose objects are in general not Hausdorff.

\begin{thm}
\label{THM-BOR}
Let $k\in\{\mathbb{R},\mathbb{C}\}$ be fixed. The category $\BOR$ of bornological (Hausdorff and non-Hausdorff) locally convex spaces over $k$, furnished with linear and continuous maps as morphisms, is integral but neither left nor right quasi-abelian.
\end{thm}
\begin{proof} 
The category $\BOR$ reflects cokernels in $\LCS$. Analogously to $\LB$, the kernel of a morphism $f\colon X\rightarrow Y$ is the inclusion $f^{-1}(0)^{\scriptscriptstyle\text{BOR}}\rightarrow X$, where the `associated bornological topology' of $f^{-1}(0)^{\scriptscriptstyle\text{BOR}}$ can be strictly finer than the topology induced by $X$; see Sieg and Wegner \cite[Exam.\ 4.1]{SiegWegner}. 
We thus get that $f$ is monic if and only if $f$ is injective, and that $f$ is epic if and only if $f$ is surjective. Consequently, pushouts and pullbacks compute algebraically precisely as in $\rMod{k}$. 
Similarly to Theorem \ref{THM-0} we conclude that $\BOR$ is integral.

A counterexample constructed by Bonet and Dierolf in \cite{BonetDierolf} (see \cite[Exam.\ 4.1]{SiegWegner}) shows that $\BOR$ is not left quasi-abelian. 
However, by Proposition \ref{prop:semi-abelian-implies-left-right-equivalence}, $\BOR$ cannot be right quasi-abelian either as $\BOR$ is semi-abelian by Proposition \ref{prop:quasi-abelian-or-integral-implies-semi-abelian}. 
Note that it was known already that $\BOR$ is semi-abelian but not quasi-abelian, cf.\ Remark \ref{REM-RAIKOV-BOR}.
\end{proof}

\begin{cor}\label{cor:integral-not-contained-in-quasi-abelian}
The class of integral categories is not contained in the class of quasi-abelian categories.
\end{cor}

We recall the connection between Ra\u{\i}kov's conjecture and the category $\BOR$.
\begin{rem}\label{REM-RAIKOV-BOR} 
Recall from \S\ref{SEC:INTRO} that \emph{Ra\u{\i}kov's conjecture} states that a category is semi-abelian if and only if it is quasi-abelian. It was posed around 1970 and answered negatively some 30 years later. Disproving it brought together aspects from algebra and analysis. 
The category $\BOR$ is one of the first two counterexamples given in the literature that falsify it. 
The other of these is due to Rump \cite[Exam.\ 1]{Rump08} and is a category of the form $A${\sf{-proj}}, where $A$ is a tilted algebra of Dynkin-type $\BE_{6}$. 
We refer to \cite{Rump11} for historical details on Ra{\u{\i}}kov's conjecture, and to \cite{Wengenroth12} for an extended survey on why the conjecture must naturally fail from the analytic point of view.
%
\end{rem}

We conclude this section on a related note. All the examples of semi-abelian categories we have studied so far (and even those in we will see in \S\ref{SEC:BVS}) are either integral or quasi-abelian. Therefore, it is natural to ask if there exists a semi-abelian category which is neither integral nor quasi-abelian. 
The authors would like to thank J.\ Wengenroth for sharing his idea behind the proof of the following result, which gives a positive answer to this question.

\begin{thm}
\label{thm:semi-category-that-is-not-integral-or-quasi-abelian}

There exist semi-abelian categories that are neither integral nor quasi-abelian. In particular, the product category $\BAN \times \BOR$ is an example of such a category.

\end{thm}

\begin{proof}

Let $\CA$ denote a semi-abelian category that is not integral (e.g.\ $\BAN$) and let $\CB$ denote a semi-abelian category that is not quasi-abelian (e.g.\ $\BOR$). Consider the \emph{product} category $\CA\times\CB$. 
The objects of $\CA\times\CB$ are pairs $(A,B)$, where $A\in\obj(\CA)$ and $B\in\obj(\CB)$, and morphisms in $\CA\times\CB$ are pairs $(f,g)$, where $f$ is a morphism in $\CA$ and $g$ is a morphism in $\CB$.

It is straightforward to check that $\CA\times\CB$ is additive and pre-abelian. In particular, (co)kernels in $\CA\times\CB$ are constructed component-wise; for example, the kernel of $(f,g)$ is 
\[
(\ker f, \ker g)\colon (\Ker f,\Ker g)\to (A,B)
\]
for $f\in\Hom_{\CA}(A,A')$ and $g\in\Hom_{\CB}(B,B')$.

As observed in \cite[pp.\ 167--168]{Rump01-almost-abelian-categories}, a pre-abelian category is semi-abelian if and only if the \emph{parallel} morphism  $h^{\sim}\colon \Coim h \to \Im h$ (that is, the canonical morphism from the coimage to the image) of a morphism $h$ is both monic and epic. It is easy to show that $(f,g)$ is monic (respectively, epic) if and only if $f,g$ are monic (respectively, epic) in their respective categories. 
Thus, since $\CA,\CB$ are both semi-abelian, the parallel morphism 
$(f,g)^{\sim} = (f^{\sim},g^{\sim})$ 
of $(f,g)$ is both monic and epic, and hence $\CA\times\CB$ is semi-abelian.

Since $\CA$ is not integral, but is semi-abelian, it cannot be left or right integral by Proposition \ref{prop:semi-abelian-implies-left-right-equivalence}. Therefore, as $\CA$ is not left integral, there is a pullback square 
\begin{equation*}
\label{eqn:pullback-square-in-semi-abelian-not-left-integral-A}
    \begin{tikzcd}
    P \arrow{d}[swap]{f'_{2}}\arrow{r}{f'_{1}}\commutes[\text{PB}]{dr}& A_{2}\arrow{d}{f_{2}} \\
    A_{1}\arrow[two heads]{r}[swap]{f_{1}} & A
    \end{tikzcd}
\end{equation*}
in $\CA$, where $f_{1}$ is an epimorphism but $f'_{1}$ is not. 
Since kernels in $\CA\times\CB$ are constructed component-wise, it follows that pullbacks are also determined by their components. Hence, we have the pullback square 
\begin{equation*}
\label{eqn:induced-pullback-square-in-AxB}
    \begin{tikzcd}
    (P,0_{\CB}) \arrow{d}[swap]{(f'_{2},0)}\arrow{r}{(f'_{1},0)}\commutes[\text{PB}]{dr}& (A_{2},0_{\CB})\arrow{d}{(f_{2},0)} \\
    (A_{1},0_{\CB})\arrow{r}[swap]{(f_{1},0)} & (A,0_{\CB})
    \end{tikzcd}
\end{equation*}
in $\CA\times\CB$, where $0_{\CB}$ is the zero object in $\CB$. Moreover, as $f_{1}\in\Hom_{\CA}(A_{1},A)$ and $0\in\Hom_{\CB}(0_{\CB},0_{\CB})$ are both epic, we see that $(f_{1},0)$ is epic; and $(f'_{1},0)$ cannot be epic since $f'_{1}$ is not. Consequently, $\CA\times\CB$ is not left integral and hence not integral.

Similarly, one can show that $\CA\times\CB$ is not quasi-abelian, and this concludes the proof.
\end{proof}

\medskip


\section{Categories of bornological vector spaces}
\label{SEC:BVS}

Below we consider categories of bornological vector spaces, in the sense introduced by Buchwalter \cite{Buchwalter} and Hogbe-Nlend \cite{Hogbe-French, Hogbe}. 
We follow the notation of Prosmans and Schneiders \cite{ProsmansSchneiders}, and consider categories whose objects are pairs $(X,\mathcal{B}_{X})$ where $X$ is a $k$-vector space and $\mathcal{B}_{X}$ is a convex bornology. 
Their morphisms are the so-called \emph{bounded} linear maps $f\colon X\rightarrow Y$, i.e.\ linear maps for which $f(B)\in\mathcal{B}_Y$ holds whenever $B\in\mathcal{B}_X$. 
See \cite[\S1]{ProsmansSchneiders} for more details. 
Notice that the term `bornological' in this section has a different meaning than in Theorem \ref{THM-BOR}. Here the bornology is an additional structure on a vector space, whereas in \S\ref{SEC:TVS} being bornological is a property that a locally convex space either enjoys or not.

We start again with identifying a category that is both quasi-abelian and integral.

\begin{thm}\label{THM-2} Let $k\in\{\mathbb{R},\mathbb{C}\}$. The category $\mathcal{B}c$ of (separated and non-separated) bornological vector spaces over $k$, furnished with bounded linear maps as morphisms, is quasi-abelian and integral.
\end{thm}
\begin{proof} By \cite[Prop.\ 1.8]{ProsmansSchneiders} the category is quasi-abelian.
By \cite[Prop.\ 1.5]{ProsmansSchneiders}, for a morphism $f\colon X\rightarrow Y$ the kernel is the inclusion map $f^{-1}(0)\rightarrow X$ and the cokernel is the quotient map $Y\rightarrow Y/\operatorname{ran}f$. Here, $f^{-1}(0)$ is endowed with the induced bornology and $Y/\operatorname{ran}f$ with the quotient bornology; see \cite[Def.\ 1.4]{ProsmansSchneiders}. 
One can now proceed as in the proof of Theorem \ref{THM-0}.
%
%
%
\end{proof}

The other two categories that are usually studied in the context of bornologies are both quasi-abelian, but neither of them is integral.

\begin{thm}
\label{THM-3}
Let $k\in\{\mathbb{R},\mathbb{C}\}$. 
The categories \vspace{-4pt}
\begin{compactitem}
\item[\emph{(i)}] $\wideparen{\mathcal{B}}c$ of separated bornological vector spaces; and\vspace{3pt}
\item[\emph{(ii)}] $\widehat{\mathcal{B}}c$ of complete bornological vector spaces,
\end{compactitem}
over $k$, furnished with bounded linear maps as morphisms, are quasi-abelian but neither left nor right integral.
\end{thm}
\begin{proof} By \cite[Prop.\ 4.10 and Prop.\ 5.6]{ProsmansSchneiders} both categories are quasi-abelian. 
If $f\colon X\rightarrow Y$ is a morphism in either one of the two categories, then its cokernel is given by the quotient map $Y\rightarrow Y/\overline{\operatorname{ran}f}$; see \cite[Prop.\ 4.6 and Prop.\ 5.6]{ProsmansSchneiders}. The \emph{closure} $\overline{\operatorname{ran}f}$ is given as the intersection of all closed subspaces $U$ of $Y$ containing $\operatorname{ran}f$. 
A subspace $U$ is \emph{closed} if limits of sequences in $U$ that converge in $X$ belong to $U$; see \cite[Def.\ 4.3]{ProsmansSchneiders}. 
Finally, convergence is defined as follows: $(x_n)_{n\in\mathbb{N}}\subseteq X$ \emph{converges to} $x\in X$ if there exists an absolutely convex set $B\in\mathcal{B}_X$, such that $(x_n)_{n\in\mathbb{N}}$ converges to $x$ in the normed space
\[
X_{B}=(\operatorname{span} B,\|\cdot\|_{B}) \; \text{ where } \; \|x\|_{B}=\inf\{\lambda>0\:|\:x\in\lambda{}B\};
\]
see \cite[Def.\ 4.1]{ProsmansSchneiders}.

Assume now that $X=(X,\|\cdot\|)$ is a Banach space that we furnish with the bornology $\mathcal{B}$ of norm-bounded sets; see, for example, \cite[p.\ 21]{Hogbe}. 
Then $(X,\mathcal{B})$ is an object of both $\wideparen{\mathcal{B}}c$ and  $\widehat{\mathcal{B}}c$. 
Moreover, a sequence $(x_n)_{n\in\mathbb{N}}\subseteq X$ converges in norm to $x\in X$ if and only if $(x_n)_{n\in\mathbb{N}}$ converges to $x$ with respect to $\mathcal{B}$. 
Indeed, if $(x_n)_{n\in\mathbb{N}}$ converges to $x$ in norm, then we choose $B$ to be the unit ball of $X$ and in view of $(X,\|\cdot\|)=(X_B,\|\cdot\|_B)$ we get convergence in bornology. Conversely, if $(x_n)_{n\in\mathbb{N}}$ converges to $x$ in some $X_B$ for $B\in\mathcal{B}$ absolutely convex, we conclude that $(x_n)_{n\in\mathbb{N}}$ converges to $x$ in norm from the fact that the inclusion $(X_B,\|\cdot\|_B)\rightarrow (X,\|\cdot\|)$ is continuous; see \cite[p.\ 282]{MeiseVogt}.

Now consider the maps $i\colon\ell^1\rightarrow\czero$ and $\Sigma\colon\ell^1\rightarrow k$ from the proof of Theorem \ref{THM-1} in $\wideparen{\mathcal{B}}c$. 
This is possible by the above and since continuous linear maps between Banach spaces send bounded sets to bounded sets. 
In view of the first part of this proof the pushout of $i$ along $\Sigma$ in $\wideparen{\mathcal{B}}c$ is given by
\begin{equation*}\label{A}
\begin{tikzcd}[column sep=21pt,row sep=21pt]
\ell^1\arrow[d, swap, "\Sigma"]\arrow[r, "i"]& \czero \arrow[d, "q\hspace{0.75pt}\circ\hspace{0.75pt}i_1"]\\
k\arrow[r, swap, "q\hspace{0.75pt}\circ\hspace{0.75pt}i_2"]& P
\end{tikzcd}
\end{equation*}
where $P=(\czero\oplus\,k)/{\textstyle\overline{\operatorname{ran}\scriptstyle\columnvec{\phantom{\shortminus}i}{\shortminus\Sigma}}}$ coincides, as a vector space, with the space $P$ from the proof of Theorem \ref{THM-1}. 
Thus, in the above diagram, $i$ is injective and $q\circ i_2=0$. By \cite[Prop.\ 4.6]{ProsmansSchneiders}, in $\wideparen{\mathcal{B}}c$ the kernel of a morphism is the preimage of zero endowed with the induced bornology. 
Thus, $i$ is monic but its pushout is not.

To complete the proof it is enough to observe that the preceding paragraph can be repeated verbatim for $\widehat{\mathcal{B}}c$.
\end{proof}

\medskip


\section{Projectives and injectives}
\label{SEC:PROJECTIVES-INJECTIVES}

Projective objects in an arbitrary category generalise the notion of projective modules arising in algebra. 
As such, they have become important objects of study in homological algebra. 
However, suitable notions of projectivity have also been studied in the categories we have seen so far. 
We focus on projectivity and leave the dual notions related to injectivity to the reader. 
Let $\CA$ be a locally small category. 
An object $P\in\CA$ is called \emph{projective} if, for every epimorphism $f\colon X\to Y$  the induced map $\Hom_{\CA}(P,f)\colon \Hom_{\CA}(P,X)\to\Hom_{\CA}(P,Y)$ is surjective.
We say $\CA$ has \emph{enough projectives} if for each $A\in\CA$ there is an epimorphism $P\to A$ with $P$ projective. 

In addition to the above, the following concept has been introduced by Osborne \cite{Osborne}, in order to address the fact that in non-abelian categories the classes of epimorphisms and cokernels do not coincide.
\begin{defn}
\label{def:quasi-projective}
\cite[Def.\ 7.52]{Osborne} 
Let $\CA$ be a pre-abelian category. 
An object $P\in\CA$ is called \emph{quasi-projective} if, for every cokernel $f\colon X\to Y$, the map $\Hom_{\CA}(P,f)$ is surjective. 
We say $\CA$ has \emph{enough quasi-projectives} if for each $A\in\CA$ there is a cokernel $P\to A$ with $P$ quasi-projective.
\end{defn}

We now use a connection between the notions of \S\ref{SEC:PRE-ABELIAN-CATEGORIES} and the ones just introduced, in order to derive some interesting consequences of our main results. 
For Proposition \ref{prop:enough-quasi-projectives-implies-left-quasi-abelian-enough-projs-implies-integral}(i) notice that in \cite{Rump01-almost-abelian-categories} the phrase `has strictly enough projectives' is equivalent to the phrase `has enough quasi-projectives' that we use here. 
\begin{prop}
\label{prop:enough-quasi-projectives-implies-left-quasi-abelian-enough-projs-implies-integral} 
Suppose $\CA$ is a pre-abelian category.\vspace{-4pt}
\begin{compactitem}
    \item[\emph{(i)}] \cite[Prop.\ 11]{Rump01-almost-abelian-categories} If $\CA$ has enough quasi-projectives (respectively, quasi-injectives), then $\CA$ is left (respectively, right) quasi-abelian.\vspace{3pt}
    \item[\emph{(ii)}] \cite[Prop.\ 3.9]{BuanMarsh} 
    If $\CA$ has enough projectives (respectively, injectives), then $\CA$ is left (respectively, right) integral.
\end{compactitem}
\end{prop}
Suppose $\CA$ is a pre-abelian category. 
Although being projective implies being quasi-projective, having enough projectives does not necessarily imply having enough quasi-projectives for $\CA$; see \cite[pp.\ 242--243]{Osborne}. 
In particular, this means that the conclusion of Proposition \ref{prop:enough-quasi-projectives-implies-left-quasi-abelian-enough-projs-implies-integral}(i) cannot be included in the conclusion of Proposition \ref{prop:enough-quasi-projectives-implies-left-quasi-abelian-enough-projs-implies-integral}(ii).

\begin{thm}
\label{thm:categories-dont-have-enough-quasi-projs-injs}
The following statements hold.\vspace{-4pt}
\begin{compactitem}
    \item[\emph{(i)}] The categories 
    $\BAN$, 
    $\NOR$, 
    $\FRE$, 
    $\HDLCS$, 
    $\HDTVS$, 
    $\NUC$, 
    $\NUCFRE$, 
    $\FS$, 
    $\FH$, 
    $\wideparen{\mathcal{B}}c$ and 
    $\widehat{\mathcal{B}}c$ 
    have neither enough projectives nor enough injectives.\vspace{3pt} 
    \item[\emph{(ii)}] The category $\BOR$ has neither enough quasi-projectives nor enough quasi-injectives.\vspace{3pt} 
    \item[\emph{(iii)}] The category 
    $\COM$ has neither enough quasi-projectives, nor enough projectives, nor enough injectives.\vspace{3pt} 
    \item[\emph{(iv)}] The category 
    $\LB$ has neither enough quasi-injectives, nor enough projectives, nor enough injectives.
\end{compactitem}
\end{thm}
\begin{proof}
(i):\; Let $\CA\in\{\BAN,\NOR,\FRE,\HDLCS,\HDTVS,\NUC,\NUCFRE,\FS,\FH\}$. 
Then $\CA$ is quasi-abelian by Theorem \ref{THM-1}, and so semi-abelian by Proposition \ref{prop:quasi-abelian-or-integral-implies-semi-abelian}. 
Thus, left integrality is equivalent to right integrality for $\CA$ by Proposition \ref{prop:semi-abelian-implies-left-right-equivalence}. 
But $\CA$ is not integral by Theorem \ref{THM-1},  and so cannot have either enough projectives or enough injectives by Proposition \ref{prop:enough-quasi-projectives-implies-left-quasi-abelian-enough-projs-implies-integral}. For $\CA\in\{\wideparen{\mathcal{B}}c, \widehat{\mathcal{B}}c\}$ one can argue analogously by employing corresponding results from \S\ref{SEC:BVS}.

(ii):\; Similar to (i), using Theorem \ref{THM-BOR}.

(iii):\; The category $\COM$ is neither left quasi-abelian, nor left integral, not right integral by Theorem \ref{THM-COM}. 
Thus, by Proposition \ref{prop:enough-quasi-projectives-implies-left-quasi-abelian-enough-projs-implies-integral}, $\COM$ can have neither enough quasi-projectives, nor enough projectives, nor enough injectives.

(iv):\; Similar to (iii), using Theorem \ref{THM-LB}.
\end{proof}
We mention that for some of the quasi-abelian categories we have seen so far, it has been previously established whether or not they have enough quasi-projectives or quasi-injectives. 
Indeed, $\BAN$, $\FRE$ and $\LCS$ have enough quasi-injectives; see \cite[Thm.\ 2.2.1]{Wengenroth03}. 
Moreover, $\CB c$, $\wideparen{\mathcal{B}}c$ and $\widehat{\mathcal{B}}c$ have enough quasi-projectives; see \cite[Prop.\ 2.13, Prop.\ 4.11 and Prop.\ 5.8]{ProsmansSchneiders}. 
Finally, $\LCS$ does not have enough quasi-projectives; see Ge{\u{\i}}ler \cite{Geiler}. 

We remark also that in the references just cited, the term `projective' is used to mean what we call quasi-projective. 
Furthermore, in a quasi-abelian category, an object is quasi-projective if and only if it is  `projective' in the sense of B{\"u}hler \cite[Def. 11.1]{Buhler}. 

We refer the reader to the diagram on p.\ \pageref{PIC} for a graphic summary of all the examples that we have studied in this article. 

%
%

\medskip

\section{The admissible intersection property}
\label{SEC:AIP}

Let $\mathcal{A}$ be a pre-abelian category. We say that $\mathcal{A}$ has \emph{admissible intersections} if there exists an exact structure $\mathcal{E}$ on $\mathcal{A}$ such that for any admissible monomorphisms 
$c\colon B \rightarrowtail D$ and $d\colon C \rightarrowtail D$, in the pullback diagram
\[
\begin{tikzcd}
A\arrow{r}{a}\arrow{d}[swap]{b}\commutes[\mathrm{PB}]{dr}& B \arrow[tail]{d}{c}\\
C \arrow[tail]{r}[swap]{d} & D
\end{tikzcd}
\]
in $\CA$, the morphisms $a$ and $b$ are also admissible monomorphisms. This property is introduced by Hassoun and Roy in \cite{HR} and has been recently considered by Br{\"u}stle, Hassoun and Tattar in \cite[\S4]{BHT}, where they showed that if $\mathcal{A}$ has admissible intersections, then $\mathcal{A}$ is quasi-abelian. 
We prove here that the converse also holds, and hence together a new characterisation of quasi-abelian categories is established. 
For the convenience of the reader, and with the kind permission of the authors of \cite{BHT}, we also include their part of the proof below.

\begin{thm}[Br{\"u}stle, Hassoun, Shah, Tattar, Wegner]
\label{thm:AI-iff-quasi-abelian}
A pre-abelian category $\mathcal{A}$ is quasi-abelian if and only if it has admissible intersections.
\end{thm}
\begin{proof} 
$(\Longrightarrow)$\; Let $\CA$ be a quasi-abelian category. 
Endowing it with the class $\CE$ of all kernel-cokernel pairs in $\CA$ yields an exact category $(\CA,\CE)$ as $\CA$ is quasi-abelian; see \cite[Rmk.\ 1.1.11]{Schneiders}. The class of admissible monomorphisms in $(\CA,\CE)$ is thus precisely the class of kernels in $\CA$. Let $c\colon B \rightarrowtail D$ and $d\colon C \rightarrowtail D$ be arbitrary admissible monomorphisms in $(\CA,\CE)$, i.e.\ $c,d$ are kernels. Then in the pullback diagram 
\[
\begin{tikzcd}
A\arrow{r}{a}\arrow{d}[swap]{b}\commutes[\mathrm{PB}]{dr}& B \arrow[tail]{d}{c}\\
C \arrow[tail]{r}[swap]{d} & D 
\end{tikzcd}
\]
the morphisms $a$ and $b$ are also kernels in $\CA$ by the dual of Kelly \cite[Prop.\ 5.2]{Kelly}. That is, $a,b$ are admissible monomorphisms, and we see that $\CA$ has admissible intersections.

$(\Longleftarrow)$\; Conversely, suppose $\CA$ has admissible intersections and let $\CE$ be an exact structure on $\CA$ witnessing this. 
We claim that $\mathcal{E}$ coincides with the class of all kernel-cokernel pairs in $\mathcal{A}$. 
%
Assume for contradiction that $A\stackrel{\scriptscriptstyle f}{\rightarrow}B\stackrel{\scriptscriptstyle g}{\rightarrow}C$ is a kernel-cokernel pair not belonging to $\CE$. Then the morphisms 
$
{\scriptstyle\columnvec{1}{g}}
 \colon B\rightarrow B\oplus C \;\text{ and }\; {\scriptstyle\columnvec{1}{0}}\colon B\rightarrow B\oplus C
$
are both sections, and thus admissible monomorphisms. 
The pullback of these two morphisms is given by
\[
\begin{tikzcd}
A\arrow{r}{f}\arrow{d}[swap]{f}\commutes[\mathrm{PB}]{dr}& B \arrow[tail]{d}{{\textstyle\columnvec{1}{g}}}\\
B \arrow[tail]{r}[swap]{{\textstyle\columnvec{1}{0}}} & B\oplus C
\end{tikzcd}
\]
Thus, we conclude that $f$ is an admissible monomorphism since $\mathcal{A}$ has admissible intersections. 
Contradiction. 
Hence, $\CE$ must contain all kernel-cokernel pairs, and so every (co)kernel is admissible. Finally, using the axioms for an exact category (see e.g.\ \cite[Def.\ 2.1]{Buhler}), we see that in $\CA$ kernels are stable under pushout and cokernels are stable under pullback, i.e.\ $\CA$ is quasi-abelian.
\end{proof}
%


\begin{acknowledgements}
The authors would like to thank Thomas Br{\"u}stle and Aran Tattar for sharing their results whilst \cite{BHT} was still in progress. 
This exchange of ideas was the trigger for the current paper. The authors would also like to thank the referee for their very helpful comments and suggestions. 
Furthermore, the first author is supported by the `th{\'e}sards {\'e}toiles' scholarship of the ISM, Bishop's University and  Universit{\'e} de Sherbrooke. 
The second author is grateful for financial support from the EPSRC grant EP/P016014/1 `Higher Dimensional Homological Algebra'. 
The third author would like to express his gratitude to the Institut des Hautes {\'E}tudes Scientifique for their hospitality during a stay in the summer 2019, where part of this research began and to Teesside University, where he carried out a majority of the work on this project.
\end{acknowledgements}

{\setstretch{1}

\begin{thebibliography}{10}

\bibitem{BanicaPopescu}
C.~B{\u{a}}nic{\u{a}} and N.~Popescu.
\newblock {\em Sur les cat{\'{e}}gories pr{\'{e}}ab{\'{e}}liennes}.
\newblock {Rev.\ Roumaine Math.\ Pures Appl.}, {\bf 10}:621--633, 1965.

\bibitem{BonetDierolf}
J.~Bonet and S.~Dierolf.
\newblock {\em The pullback for bornological and ultrabornological spaces}.
\newblock {Note Mat.}, {\bf 25}(1):63--67, 2005/06.

\bibitem{BHLR-reduction}
T.~Br\"{u}stle, S.~Hassoun, D.~Langford, and S.~Roy.
\newblock {\em Reduction of exact structures}.
\newblock {J.\ Pure Appl.\ Algebra}, {\bf 224}(4):106212, 29, 2020.

\bibitem{BHT}
T.~Br{\"u}stle, S.~Hassoun, and A.~Tattar.
\newblock {\em Intersections, sums, and the Jordan-H{\"{o}}lder property for exact categories}.
\newblock Preprint, 2020.
\newblock \url{https://arxiv.org/abs/2006.03505}.

\bibitem{BuanMarsh}
A.~B.~Buan and R.~J.~Marsh.
\newblock {\em From triangulated categories to module categories via
  localization {II}: calculus of fractions}.
\newblock {J.\ Lond.\ Math.\ Soc.\ (2)}, {\bf 86}(1):152--170, 2012.

\bibitem{Buchwalter}
H.~Buchwalter.
\newblock {\em Espaces vectoriels bornologiques}.
\newblock {Publ.\ D{\'{e}}p.\ Math.\ (Lyon)}, {\bf 2}:1--53, 1965.

\bibitem{Buhler}
T.~B{\"{u}}hler.
\newblock {\em Exact categories}.
\newblock {Expo.\ Math.}, {\bf 28}(1):1--69, 2010.

\bibitem{Enomoto}
H.~Enomoto.
\newblock {\em The {J}ordan-{H}{\"{o}}lder property and {G}rothendieck monoids
  of exact categories}.
\newblock Preprint, 2019.
\newblock \url{https://arxiv.org/abs/1908.05446}.

\bibitem{FrerickSieg}
L.~Frerick and D.~Sieg.
\newblock {\em Exact {C}ategories in {F}unctional {A}nalysis}.
\newblock Lecture notes, 2010.
\newblock \url{https://www.math.uni-trier.de/abteilung/analysis/HomAlg.pdf}.

\bibitem{GabrielZisman}
P.~Gabriel and M.~Zisman.
\newblock {\em Calculus of fractions and homotopy theory}.
\newblock Ergebnisse der Mathematik und ihrer Grenzgebiete, Band 35.
  Springer-Verlag New York, Inc., New York, 1967.

\bibitem{Geiler}
V.~A.~Ge{\u{\i}}ler.
\newblock {\em The projective objects in the category of locally convex
  spaces}.
\newblock {Funkcional.\ Anal.\ i Prilo{\v{z}}en.}, {\bf 6}(2):79--80, 1972.

\bibitem{HR}
S.~Hassoun and S.~Roy.
\newblock {\em Admissible intersection and sum property}.
\newblock Preprint, 2019.
\newblock \url{https://arxiv.org/abs/1906.03246}.

\bibitem{HassounShah}
S.~Hassoun and A.~Shah.
\newblock {\em Integral and quasi-abelian hearts of twin cotorsion pairs on
  extriangulated categories}.
\newblock {Commun.\ Algebra}, {\bf 48}(12):5142--5162, 2020.

\bibitem{Hogbe-French}
H.~{Hogbe-Nlend}.
\newblock {\em Th{\'{e}}orie des bornologies et applications}.
\newblock Lecture Notes in Mathematics, Vol.\ 213. Springer-Verlag, Berlin-New
  York, 1971.

\bibitem{Hogbe}
H.~{Hogbe-Nlend}.
\newblock {\em Bornologies and functional analysis}.
\newblock North-Holland Publishing Co., Amsterdam-New York-Oxford, 1977.

\bibitem{JanelidzeMarkiTholen}
G.~Janelidze, L.~M{\'{a}}rki, and W.~Tholen.
\newblock {\em Semi-abelian categories}.
\newblock {J.\ Pure Appl.\ Algebra}, {\bf 168}(2-3):367--386, 2002.

\bibitem{Kelly}
G.~M.~Kelly.
\newblock {\em Monomorphisms, epimorphisms, and pull-backs}.
\newblock {J.\ Austral.\ Math.\ Soc.}, {\bf 9}:124--142, 1969.

\bibitem{KopylovWegner12}
Ya.~Kopylov and S.-A.~Wegner.
\newblock {\em On the notion of a semi-abelian category in the sense of
  {P}alamodov}.
\newblock {Appl.\ Categ.\ Structures}, {\bf 20}(5):531--541, 2012.

\bibitem{LiuY}
Y.~Liu.
\newblock {\em Hearts of twin cotorsion pairs on exact categories}.
\newblock {J.\ Algebra}, {\bf 394}:245--284, 2013.

\bibitem{LiuYNakaoka}
Y.~Liu and H.~Nakaoka.
\newblock {\em Hearts of twin cotorsion pairs on extriangulated categories}.
\newblock {J.\ Algebra}, {\bf 528}:96--149, 2019.

\bibitem{MeiseVogt}
R.~Meise and D.~Vogt.
\newblock {\em Introduction to functional analysis}, volume~2 of {\em Oxford
  Graduate Texts in Mathematics}.
\newblock The Clarendon Press, Oxford University Press, New York, 1997.
\newblock Translated from the German by M.\ S.\ Ramanujan and revised by the
  authors.

\bibitem{Nakaoka}
H.~Nakaoka.
\newblock {\em General heart construction for twin torsion pairs on
  triangulated categories}.
\newblock {J.\ Algebra}, {\bf 374}:195--215, 2013.

\bibitem{NakaokaPalu}
H.~Nakaoka and Y.~Palu.
\newblock {\em Extriangulated categories, {H}ovey twin cotorsion pairs and
  model structures}.
\newblock {Cah.\ Topol.\ G{\'{e}}om.\ Diff{\'{e}}r.\ Cat{\'{e}}g.}, {\bf
  60}(2):117--193, 2019.

\bibitem{Osborne}
M.~S.~Osborne.
\newblock {\em Basic homological algebra}, volume 196 of {\em Graduate Texts in
  Mathematics}.
\newblock Springer-Verlag, New York, 2000.

\bibitem{Palamodov68}
V.~P.~Palamodov.
\newblock {\em The projective limit functor in the category of topological
  linear spaces}.
\newblock {Mat.\ Sb.\ (N.S.)}, {\bf 75 (117)}:567--603, 1968.

\bibitem{Palamodov71}
V.~P.~Palamodov.
\newblock {\em Homological methods in the theory of locally convex spaces}.
\newblock {Uspehi Mat.\ Nauk}, {\bf 26}(1(157)):3--65, 1971.

\bibitem{Prosmans95}
F.~Prosmans.
\newblock {\em Alg{\`e}bre {H}omologique {Q}uasi-{A}b{\'e}lienne}.
\newblock M{\'e}moire de DEA, Universit{\'e} Paris 13, Villetaneuse (France),
  1995.
\newblock \url{http://users.belgacom.net/fprosmans/dea.pdf}.

\bibitem{Prosmans00}
F.~Prosmans.
\newblock {\em Derived categories for functional analysis}.
\newblock {Publ.\ Res.\ Inst.\ Math.\ Sci.}, {\bf 36}(1):19--83, 2000.

\bibitem{ProsmansSchneiders}
F.~Prosmans and J.-P.~Schneiders.
\newblock {\em A homological study of bornological spaces}.
\newblock {Pr{\'{e}}publications Ma\-th{\'{e}}\-ma\-ti\-ques de
  l'Universit{\'{e}} Paris 13}, {\bf 00-21}, 2000.

\bibitem{Quillen}
D.~Quillen.
\newblock {\em Higher algebraic {$K$}-theory. {I}}.
\newblock In H.~Bass, editor, {\em Algebraic K-Theory I. Proceedings of the
  Conference Held at the Seattle Research Center of Battelle Memorial
  Institute, August 28 -- September 8, 1972}, pages 85--147. Springer, Berlin,
  1973.

\bibitem{Raikov}
D.~A.~Ra\u{\i}kov.
\newblock {\em Semiabelian categories}.
\newblock {Dokl.\ Akad.\ Nauk SSSR}, {\bf 188}:1006--1009, 1969.

\bibitem{Rump01-almost-abelian-categories}
W.~Rump.
\newblock {\em Almost abelian categories}.
\newblock {Cah.\ Topol.\ G{\'{e}}om.\ Diff{\'{e}}r.\ Cat{\'{e}}g.}, {\bf 42}(3):163--225, 2001.

\bibitem{Rump01-star-modules-tilting}
W.~Rump.
\newblock {\em {$\ast$}-modules, tilting, and almost abelian categories}.
\newblock {Comm.\ Algebra}, {\bf 29}(8):3293--3325, 2001.

\bibitem{Rump04-categories-of-lattices}
W.~Rump.
\newblock {\em Categories of lattices, and their global structure in terms of
  almost split sequences}.
\newblock {Algebra Discrete Math.}, {\bf 1}:87--111, 2004.

\bibitem{Rump04-differentiation}
W.~Rump.
\newblock {\em Differentiation for orders and {A}rtinian rings}.
\newblock {Algebr.\ Represent.\ Theory}, {\bf 7}(4):395--417, 2004.

\bibitem{Rump06}
W.~Rump.
\newblock {\em Global theory of lattice-finite {N}oetherian rings}.
\newblock {Algebr.\ Represent.\ Theory}, {\bf 9}(3):227--239, 2006.

\bibitem{Rump08}
W.~Rump.
\newblock {\em A counterexample to {R}aikov's conjecture}.
\newblock {Bull.\ Lond.\ Math.\ Soc.}, {\bf 40}(6):985--994, 2008.

\bibitem{Rump11}
W.~Rump.
\newblock {\em Analysis of a problem of {R}aikov with applications to barreled
  and bornological spaces}.
\newblock {J.\ Pure Appl.\ Algebra}, {\bf 215}(1):44--52, 2011.

\bibitem{Schneiders}
J.-P.~Schneiders.
\newblock {\em Quasi-abelian categories and sheaves}.
\newblock {M{\'{e}}m.\ Soc.\ Math.\ Fr.\ (N.S.)}, {\bf 76}:vi+134, 1999.

\bibitem{Shah}
A.~Shah.
\newblock {\em Quasi-abelian hearts of twin cotorsion pairs on triangulated
  categories}.
\newblock {J.\ Algebra}, {\bf 534}:313--338, 2019.

\bibitem{SiegWegner}
D.~Sieg and S.-A.~Wegner.
\newblock {\em Maximal exact structures on additive categories}.
\newblock {Math.\ Nachr.}, {\bf 284}(16):2093--2100, 2011.

\bibitem{Wegner}
S.-A.~Wegner.
\newblock {\em The heart of the {B}anach spaces}.
\newblock {J.\ Pure Appl.\ Algebra}, {\bf 221}(11):2880--2909, 2017.

\bibitem{Wengenroth03}
J.~Wengenroth.
\newblock {\em Derived functors in functional analysis}, volume 1810 of {\em
  Lecture Notes in Mathematics}.
\newblock Springer-Verlag, Berlin, 2003.

\bibitem{Wengenroth12}
J.~Wengenroth.
\newblock {\em The {R}a{\u{\i}}kov conjecture fails for simple analytical
  reasons}.
\newblock {J.\ Pure Appl.\ Algebra}, {\bf 216}(7):1700--1703, 2012.

\bibitem{Yoneda61}
N.~Yoneda.
\newblock {\em On {E}xt and exact sequences}.
\newblock {J.\ Fac.\ Sci.\ Univ.\ Tokyo Sect.\ I}, {\bf 8}:507--576, 1960.

\end{thebibliography}

}
\end{document}